\pdfoutput=1

\documentclass[11pt]{amsart}

\usepackage{lmodern}
\usepackage[T1]{fontenc}
\usepackage[utf8]{inputenc}

\usepackage{mathtools}
\mathtoolsset{showonlyrefs}
\usepackage{amssymb}
\usepackage{amsthm}
\usepackage{graphicx}
\usepackage{enumitem}
\usepackage{textcomp}

\usepackage{float}
\usepackage[noend]{algpseudocode}

\usepackage{caption}
\usepackage{subcaption}
\usepackage{tikz}
\usetikzlibrary{calc,decorations.pathreplacing,shapes}
\tikzset{%
every picture/.style={line width=1pt},%
every node/.style={line width=1.5pt,circle,fill=none,draw,inner sep=0,minimum size=13pt,font=\footnotesize},%
cnode1/.style={draw=color1,line width=1.5pt,regular polygon,regular polygon sides=3,fill=color1,inner sep=0,minimum size=13pt,font=\footnotesize},%
cnode2/.style={draw=color2,line width=1.5pt,regular polygon,regular polygon sides=4,fill=color2,inner sep=0,minimum size=13pt,font=\footnotesize},%
cnode3/.style={draw=color3,line width=1.5pt,regular polygon,regular polygon sides=5,fill=color3,inner sep=0,minimum size=13pt,font=\footnotesize}%
}
\definecolor{color1}{rgb}{0.62068966,0.06896552,1}
\definecolor{color2}{rgb}{0,0.5862069,0}
\definecolor{color3}{rgb}{0.75862069,0.20689655,0}

\usepackage{microtype}
\usepackage{bookmark}
\usepackage{hyperref}

\newcommand*{\E}{\mathbb{E}}

\newcommand*{\Prob}{\mathbb{P}}
\newcommand*{\indicator}[1]{\mathbf1_{\{#1\}}}
\newcommand*{\cL}{{\mathcal L}}
\newcommand*{\hcL}{\hat{\mathcal L}}
\newcommand*{\pstar}{\vec{p}^{\hspace*{0.5ex}*}}
\newcommand*{\Dstar}{\vec{\Delta}^{*}}
\DeclareMathOperator{\aut}{Aut}

\newcommand*{\simp}[1]{#1_{\mathrm{simp}}}

\newtheorem*{theorem*}{Theorem}
\newtheorem{theorem}{Theorem}
\newtheorem{lemma}[theorem]{Lemma}
\newtheorem{cor}[theorem]{Corollary}

\newtheorem*{prop*}{Proposition}
\newtheorem{conj}[theorem]{Conjecture}

\begin{document}
\title[Counting proper colourings in 4-regular graphs]{Counting proper colourings in 4-regular graphs via the Potts model}

\author[E. Davies]{Ewan Davies}
\address{Korteweg-de Vries Institute for Mathematics, University of Amsterdam, P.O. Box 94248,  
1090 GE Amsterdam, The Netherlands}
\thanks{The research leading to these results has received funding from the European Research Council under the European Union's Seventh Framework Programme (FP7/2007-2013) / ERC grant agreement \textnumero{} 339109.}
\email{e.s.d.davies@uva.nl}
\date{\today}

\begin{abstract}
We give tight upper and lower bounds on the internal energy per particle in the antiferromagnetic $q$-state Potts model on $4$-regular graphs, for $q\ge 5$. 
This proves the first case of a conjecture of the author, Perkins, Jenssen, and Roberts, and implies tight bounds on the antiferromagnetic Potts partition function.  

The zero-temperature limit gives upper and lower bounds on the number of proper $q$-colourings of $4$-regular graphs, which almost proves the case $d=4$ of a conjecture of Galvin and Tetali.
For any $q \ge 5$ we prove that the number of proper $q$-colourings of a $4$-regular graph is maximised by a union of $K_{4,4}$'s. 
\end{abstract}

\maketitle
\thispagestyle{empty}

\section{Introduction}

The Potts model is a probabilistic model on assignments of interacting colours to the vertices of a graph. 
A temperature parameter controls how strongly the colours interact, and in the zero-temperature limit valid assignments of colours are proper $q$-colourings of the graph. 
For a graph $G$ we call a map $\sigma : V(G) \to [q]$ a $q$-colouring of $G$, and we let $m(\sigma)$ denote the number of monochromatic edges of $G$ under $\sigma$. 
The usual graph-theoretic notion of a \emph{proper colouring} is then such a map $\sigma$ with zero monochromatic edges. 
The $q$-state Potts model partition function is
\begin{align*}
Z_G^q(\beta)&= \sum_{\sigma :V(G)\to [q]} e^{ - \beta m(\sigma) }\,, 
\end{align*}
where $\beta$ is the \emph{inverse temperature} and for $\beta>0$ we say the model is \emph{antiferromagnetic}. 
For our purposes, this will mean that increasing $\beta$ biases the Potts model towards colourings with few monochromatic edges.

Using this partition function, the Potts model~\cite{potts1952some} (with no external field) is a random $q$-colouring $\sigma$ of $G$, chosen according to the distribution
\[ \sigma \mapsto \frac{ e^{-\beta m(\sigma)}   }{  Z_G^q(\beta)}\,.\]
Note that for $q=2$ this is known as the \emph{Ising model}~\cite{ising1925beitrag}. 
See~\cite{mezard2009information,wu1982potts} for more details on the Potts model. 

There is a strong link between the Potts model and extremal combinatorics, as in the limit $\beta\to +\infty$, the model converges to the uniform distribution on proper $q$-colourings of a graph, and $Z^q_G(\beta)$ tends to $c^q(G)$, the number of proper $q$-colourings of $G$.
Extremal questions on the number of $q$-colourings have received a lot of attention; see e.g.~\cite{linial1983legal,loh2010maximizing} for examples in dense graphs and~\cite{zhao2017extremal} for a survey (which also covers related problems) on regular graphs.
A conjecture of Galvin and Tetali~\cite{galvin2004weighted} states that over all $d$-regular graphs $G$ and $q\ge 3$, we have 
\begin{equation}\label{eq:Cconj}
\frac{1}{|V(G)|}\log c^q(G) \le \frac{1}{2d}\log c^q(K_{d,d})\,.
\end{equation}
Note that the case $q=2$ follows easily from the observation that $K_{d,d}$ is the smallest $d$-regular graph which is bipartite (that is, which admits a proper 2-colouring).

In this paper we obtain the above inequality via the Potts model, and a natural quantity to consider is the \emph{free energy per particle} $F^q_G(\beta) := \frac{1}{|V(G)|} \log Z^q_G(\beta)$.
The normalisation is convenient if one wishes to compare the number of $q$-colourings in graphs on different numbers of vertices. 
Inequality~\eqref{eq:Cconj} is the zero-temperature version of
\begin{equation}\label{eq:Fconj}
F^q_G (\beta ) \le F^q_{K_{d,d}}(\beta)\,,
\end{equation}
since in the limit $\beta\to+\infty$, inequality \eqref{eq:Fconj} becomes exactly~\eqref{eq:Cconj}.

For general $d$, Galvin~\cite{galvin2013maximizing,galvin2006bounding} showed that \eqref{eq:Fconj} holds for all $\beta$ when $G$ is bipartite, and \eqref{eq:Cconj} holds for all $d$ when $q\ge 2\binom{d|V(G)|/2}{4}$. 
Recently, for $d=3$ both~\eqref{eq:Cconj} and~\eqref{eq:Fconj} were shown to hold for all $q$ and $\beta>0$ by the author, Perkins, Jenssen, and Roberts~\cite{davies2016extremes}, using somewhat different techniques than in previous work.

The main difference between the techniques of~\cite{davies2016extremes} and previous work is the focus on an \emph{observable} (expectation over the Potts model) rather than the partition function itself. 
The expected number of monochromatic edges indeed an observable, and we scale by the number of vertices to obtain the \emph{internal energy per particle} $U^q_G(\beta)$, where
\begin{align*}
U^q_G(\beta)
&= \frac{1}{|V(G)|}\E  [m(\sigma)]\\
&= \frac{1}{|V(G)|} \frac{  \sum_{\sigma : V(G)\to  [q]} m(\sigma )e^{- \beta m(\sigma) } }{ Z_G^q(\beta)} \\
&= -\frac{1}{|V(G)|} \frac{1}{ Z_G^q(\beta)}\frac{\partial}{\partial\beta} Z_G^q(\beta)\\
&= -\frac{1}{|V(G)|} \frac{\partial}{\partial \beta} \log Z_G^q(\beta) = -\frac{\partial}{\partial\beta}F^q_G(\beta)\,.
\end{align*}

When $\beta=0$ there are no interactions between colours and  $Z_G^q(0) = q^{|V(G)|}$ for all $G$.  
Then to obtain the free energy per particle we can integrate the internal energy per particle,
\begin{equation}
\label{eq:integrateFormula}
F^q_G(\beta) = \log q -  \int_{0}^\beta U^q_G(t) \, dt\, .
\end{equation}

In \cite{davies2016extremes}, tight bounds on $U^q_G(\beta)$ were given for $3$-regular graphs for all $\beta>0$, leading to~\eqref{eq:Fconj} and~\eqref{eq:Cconj} for $d=3$ and all $q$, $\beta>0$. 
Here we extend the methods to $d=4$ under the additional restriction that $q\ge 5$, almost settling the $d=4$ case of \eqref{eq:Cconj}, the conjecture of Galvin and Tetali, as only the cases $q\in\{3,4\}$ remain. 

\begin{theorem}
\label{thm:d4}
For any $4$-regular graph $G$, any $q \ge 5$, and any $\beta >0$,
\[ U^q_{K_{4,4}}(\beta) \le U^q_G(\beta) \le U^q_{K_5}(\beta).\]
Equalities hold respectively if and only if $G$ is a union of $K_{4,4}$'s or a union of $K_5$'s. 
As a corollary via \eqref{eq:integrateFormula}, we have
\[ F^q_{K_5}(\beta) \le F^q_G(\beta) \le F^q_{K_{4,4}}(\beta).\]
\end{theorem}

This result gives~\eqref{eq:Cconj} by the following simple argument (also given in~\cite{davies2016extremes}).

\begin{lemma}
\label{lem:betatocolourings}
Fix positive integers $d$ and $q$.  If for all $d$-regular graphs $G$, and all $\beta >0$, we have $U^q_G( \beta) \ge U^q_{K_{d,d}}(\beta)$, then
\[ \frac{1}{|V(G)|}\log c^q(G) \le \frac{1}{2d}\log c^q(K_{d,d})\,. \]
\end{lemma}
\begin{proof}
Let $G$ be any $d$-regular graph. If $c^q(G)=0$ then the conclusion clearly holds. Otherwise, we take logarithms and write
\begin{align*}
\frac{1}{|V(G)|} \log c^q(G) &= \lim_{\beta \to +\infty} \frac{1}{|V(G)|} \log Z^q_G(\beta) \\
 &=  \log q -  \int_{0}^\infty U^q_G(\beta) \, d\beta \\
 &\le \log q - \int_{0}^\infty U^q_{K_{d,d}}(\beta) \, d\beta \\
 &= \frac{1}{2d} \log c^q(K_{d,d})\,.\qedhere
\end{align*}
\end{proof}

Similarly, Theorem~\ref{thm:d4} also implies the lower bound
\[
\frac{1}{|V(G)|}\log c^q(G)\ge \frac{1}{5}\log c^q(K_5)
\]
for $4$-regular graphs $G$ when $q\ge 5$, but the analogous bound was proved for all $d$ and $q\ge d+1$ by Bezáková, Štefankovič, Vazirani, and Vigoda~\cite{bezakova2008accelerating}.

Bounding the number of colourings of a graph via the internal energy per particle of the Potts model has another advantage that was highlighted in~\cite{davies2017tight}. 
The main advance of~\cite{davies2017tight} is a general method for turning proofs of the kind given below (and similar ones cited throughout) into stability results and tight bounds on the individual coefficients of the partition function. 
In this context, the most interesting result of these methods is a stability version of the fact that unions of $K_{4,4}$ maximise the number of proper $q$-colourings (see Lemma~\ref{lem:betatocolourings}). 
For more details, and more general results about the coefficients of the Potts model partition function, one should combine the methods of~\cite{davies2017tight} with our Theorem~\ref{thm:d4}. 

To state the stability result we require a little notation. 
When $2d$ divides $n$, let $H_{d,n}$ be the graph consisting of $n/(2d)$ copies of $K_{d,d}$, and let $\delta_{\circ}(G,H)$ denote the sampling distance between the bounded degree graphs $G$ and $H$ (see~\cite{davies2017tight,lovasz2012large}). 
This is a natural measure of how similar $G$ and $H$ are, and for $4$-regular $G$ (in lieu of greater detail) we can loosely interpret $\delta_{\circ}(G,K_{4,4})$ as the proportion of vertices in $G$ that are contained in a $K_{4,4}$.

\begin{cor}
For all $q \geq 5$, there exists $\kappa(q)>0$ such that for any $n$ divisible by $8$ and any $4$-regular graph $G$ on $n$ vertices, we have
\begin{align*}
& \frac{1}{n} \log  c^q(G) \leq  \frac{1}{n} \log c^q(H_{4,n}) - \kappa(q)\cdot\delta_{\circ}(G,K_{4,4})
\end{align*}
\end{cor}

Finally, we suggest that it is not entirely surprising that the extremal problems we discuss are harder for smaller $q$, as in general much more is known about the Potts model for large $q$. 
The methods of Galvin~\cite{galvin2013maximizing} apply when $q\ge 2\binom{d|V(G)|/2}{4}$, and for example, one can efficiently approximate the number of proper $q$-colourings when $q\ge 11d/6$~\cite{vigoda2000improved} with a suitable randomised algorithm, but this is impossible when $q<d$ unless $\mathrm{NP}=\mathrm{RP}$, as it is $\mathrm{NP}$-hard to determine even whether a single proper $q$-colouring exists (when $q\ge3$)~\cite{garey1974simplified}. 

The rest of this paper is organised as follows. 
In Section~\ref{sec:method} we describe the method by which we bound the internal energy of the Potts model, noting that symmetry in the colours reduces the problem to a finite calculation. 
In section~\ref{sec:d4computer} we describe a practical implementation of the method by computer. 
This section contains the main advance of the paper: an application of efficient graph isomorphism and generation algorithms to the problem of generating the information necessary to carry out the method of Section~\ref{sec:method}.
Finally, in Section~\ref{sec:conc} we give some concluding remarks.

The computational techniques described in this paper are implemented in Cython\footnote{http://cython.org} code using libraries from the Sage\footnote{https://www.sagemath.org} computer algebra system, version 8.0. 
The source code is included with the paper as a collection of ancillary files.

\section{The method}\label{sec:method}

The general form of the method used here has been applied to problems of a similar nature; in~\cite{davies2017independent,davies2018average,perarnau2016counting} bounds on the number of independent sets or matchings in certain classes of graphs are obtained via bounds on an observable of the hard-core or monomer-dimer models, and a similar problem for the Widom--Rowlinson model was solved in~\cite{cohen2017widom}.
One can also obtain tight bounds on the individual coefficients of the partition function in certain cases~\cite{davies2017tight}. 
We now describe the method for the Potts model, which is essentially the same in this paper and in~\cite{davies2016extremes}. 
We leave the regularity of the graph $d$ as a parameter, but will only address the case $d=4$ in the proof.

\subsection{Outline}

Fix a $d$-regular graph $G$, positive integer $q$, and $\beta >0$. 
To bound $U^q_G(\beta)$, we considering the following experiment.
Let $v$ be a vertex of $G$ chosen uniformly at random, and independently sample a colouring $\sigma:V(G)\to [q]$ from the Potts model on $G$.
For each neighbour $u$ of $v$, record the number of its \emph{external neighbours} (neighbours outside $v \cup N(v)$) receiving each colour and record also any edges within $N(v)$. 
This gives a \emph{local view} of $\sigma$ from $v$.
Note that although we have sampled a colouring $\sigma$ of the whole graph, the colours of $v$ and its neighbours are not revealed. 
The graph structure amongst the external neighbours is not required, and so formally the local view is a graph on $d+1$ vertices such that for $d$ of the vertices we have a multiset of externally observed colours, see Figure~\ref{fig:L}. 

\begin{figure}[h]
\begin{tikzpicture}
	\def\bcs{{{1,2},{1,3},{1,1,2},{1,3,3}}}
	\def\bcss{{"1,2","1,3","1,1,2","1,3,3"}}
	\node (v) at (0,0) {$v$};
	
	\pgfmathsetmacro{\num}{4}
	\pgfmathsetmacro{\spa}{2.5}
	\pgfmathsetmacro{\len}{1.5}
	\foreach \i in {1,...,\num} {
		\pgfmathsetmacro{\spai}{(\i-1-(\num-1)/2)*\spa}
		\pgfmathsetmacro{\multiseti}{\bcss[\i-1]}
		\node[label=below:$\mathclap{[\multiseti]}$] (u\i) at ($ (\spai,-\len) $) {};
		\draw (v) -- (u\i);
	}
	\draw (u1) -- (u2);	

\end{tikzpicture}
\caption{A local view where the multisets of observed colours are shown below each neighbour of $v$.\label{fig:L}}
\end{figure}
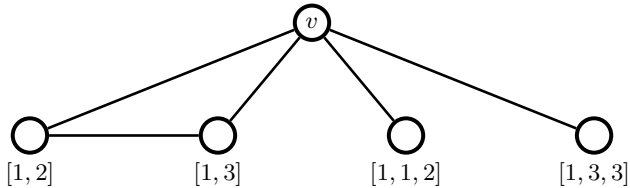

An important part of the method is that, conditioned on the local view, the distribution of colourings of $v$ and its neighbours can be determined. 
This is a fundamental feature of the Potts model (and more generally any \emph{Gibbs measure}) known as the spatial Markov property, which can easily be derived from the definition of the distribution and partition function.
In our case, for any $v$ and the random colouring $\sigma$, conditioned on the colours of the external neighbours of $v$, the distribution of the colours on $\{v\}\cup N(v)$ is independent of the colours of any vertex at distance greater than $2$ from $v$. 

For fixed $d$ and $q$ there are is a finite set of possible local views which we denote $\cL_{d,q}$. 
By the above experiment, each $d$-regular graph $G$ and inverse temperature $\beta$ induces a probability distribution on $\cL_{d,q}$.
Not all probability distributions on $\cL_{d,q}$ can arise from a graph; there are certain consistency conditions that must hold. 
Here we consider a family of conditions identified in~\cite{davies2016extremes} but not needed for the case $d=3$.
Let $\mathcal S_{d,q}$ be the set of \emph{$q$-partitions} of $d$; that is, partitions of the form $d=a_1+a_2+\dotsb+a_k$ where $a_i\ge 1$ are positive integers and $k\le q$. 
Let $x\in V(G)$, and note that any $q$-colouring of $N(x)$ induces a $q$-partition of $d$ which we denote $H(x)$ and compute by counting the frequencies of each colour on $N(x)$. 
For example, if $N(x)$ is assigned the multiset of colours $[1,2,4,4]$, then $H(x) = 2+1+1$.  
The family of consistency conditions that we use to prove the lower bound in Theorem~\ref{thm:d4} is that for every $S \in \mathcal S_{d,q}$, the probability that for $v$ as in the local view, $\Prob(H(v)=S)$ is the same as the average over $u\in N(v)$ of $\Prob(H(u)=S)$. 
This holds for in any regular graph since both $v$ and a uniformly random neighbour are distributed uniformly over $V(G)$.
Note that because $a_i\ge 1$ in any partition $d=a_1+a_2+\dotsb+a_k$, we must have $k\le 4$, and hence $\mathcal S_{4,4}=\mathcal S_{4,q}$ for any $q\ge 5$. 
This simply means that our set of consistency conditions is bounded in size independent of $q$, which is important for our calculations.
These consistency conditions are the first way in which the method of this paper and the method of~\cite{davies2016extremes} differ. 
For the $d=3$ case addressed in~\cite{davies2016extremes}, simpler consistency conditions suffice which we do not discuss here. 

It turns out that $U^q_G(\beta)$ and the probabilities one must compute for these consistency conditions are linear functions of the probabilities that each local view appears in the random experiment. 
Then to maximise or minimise $U_G^q(\beta)$ over all $d$-regular graphs, we consider a linear programming relaxation of the problem, optimising over all probability distributions on $\cL_{d,q}$ that satisfy the above consistency conditions. 
For each fixed $d$, $q$, and $\beta$, this is a finite linear program with $|\cL_{d,q}|$ variables, but we wish to solve an infinite family of these programs indexed by $q$ and $\beta$, hence we cannot yet reach for a computer.

To prove Theorem~\ref{thm:d4} we show, separately for the minimum and the maximum, that the distributions corresponding to $K_{4,4}$ and $K_5$ respectively are feasible in the dual linear program. 
In~\cite{davies2016extremes} it was conjectured that when $d\ge 4$ and $q \ge d+1$, the linear program above is indeed minimised and maximised by distributions corresponding to $K_{d,d}$ and $K_{d+1}$ respectively, hence this paper confirms that conjecture for $d=4$.

\subsection{Linear programs}\label{sec:d4proof}

In this section we describe the linear programs which are used to prove Theorem~\ref{thm:d4}. 
Recall the process for sampling a local view: draw a colouring $\sigma: V(G)\to [q]$ according to the $q$-colour Potts model with inverse temperature $\beta$ and independently, uniformly at random choose a vertex $v\in V(G)$. 
The random local view $L$ consists of the induced subgraph of $G$ on $\{v\} \cup N(v)$, together with, for each $u\in N(v)$, the multiset of colours that appears in $N(u) \setminus \big(\{v\} \cup N(v)\big)$. 

To make the program of finite size independent of $q$ we consider equivalence classes of the local views under permutations of the colours. 
There are at most $d(d-1)$ vertices whose colours we record to form the local view, and all of the quantities we need to calculate are invariant under colour permutations, hence for $d=4$ we need only consider $12$ colours in the linear program. 
Let $\hcL_{d,q}$ be a subset of the possible local views $\cL_{d,q}$ that contains exactly one representative from each equivalence class. 
The construction of $\hcL_{4,q}$, which contains $3529$ local views, is one of the computational problems that required a nontrivial implementation to be achievable in a short time. 
See Section~\ref{sec:d4computer} for details.

Suppose that the local view $L$ arises from selecting the colouring $\sigma$ and the vertex $v$. 
We refer to the coloured external neighbours at distance two from $v$ as the \emph{boundary}, and write $V_L$ for the set of uncoloured vertices in $L$. 
The colouring $\sigma$ induces a \emph{local colouring} $\chi : {V_L}\to [q]$ that, by the spatial Markov property of the Potts model, is distributed according to the Potts model on $L$ (which includes some fixed colours on the boundary). 
For $\chi: V_L\to [q]$, write $m(\chi)$ for the total number of monochromatic edges in $L$ (including any monochromatic edges between $V_L$ and the boundary), and write $m_v(\chi)$ for the number of monochromatic edges in $L$ incident to $v$. 
Then, with the local partition function defined as
\begin{align}
Z_L^q(\beta) &:= \sum_{\chi: V_L\to [q]}e^{-\beta m(\chi)}\,,
\end{align}
the random local colouring $\chi: V_L \to [q]$ is distributed according to 
\[
\chi\mapsto\frac{e^{-\beta m(\chi)}}{Z_L^q(\beta)}\,.
\]

We can now write $U^q_G(\beta)$ as an expectation over the random local view $L$.
Each edge of $G$ is incident to exactly two vertices, hence
\begin{align}
U^q_G(\beta) &= \frac{1}{|V(G)|} \E_\sigma[m(\sigma)]\\
&= \frac{1}{2|V(G)|} \sum_{v\in V(G)}\sum_{u\in N(v)} \Prob_\sigma(\text{$uv$ monochromatic})\\
&= \frac{1}{2}\E_L\Big[\sum_{u\in N(v)} \Prob_\chi(\text{$uv$ monochromatic}|L)\Big]\\
&= \frac{1}{2}\E_{L,\chi}[m_v(\chi)]\,,
\end{align}
where the subscripts show which random variables are involved in the expectations and probabilities.

We define the vector $\vec{c}\in \mathbb{R}^{\hcL_{4,q}}$ with components
\begin{equation}\label{eq:cLdef}
\vec{c}_L := \frac{1}{2Z^q_L}\sum_{\chi:V_L\to[q]} m_v(\chi) e^{-\beta m(\chi)}\,,
\end{equation}
which is the objective for both our linear programs. 

To constrain the linear programs, recall that we consider the $q$-partitions $H(v)$ and $H(u)$ that appear in the local view. 
For $S\in\mathcal S_{4,q}$, we have
\begin{align}
\E_{L,\chi}\big[\indicator{H(v)=S}\big] = \E_{L,\chi}\left[\frac{1}{d}\sum_{u\in N(v)}\indicator{H(u)=S}\right]\,,
\end{align}
and for use in the linear program we define vectors $\vec\gamma^S\in \mathbb{R}^{\hcL_{4,q}}$ for $S\in\mathcal S_{4,q}$ with components
\begin{equation}\label{eq:gamLdef}
\vec\gamma^S_L := \sum_{\chi:V_L\to[q]}\left(\indicator{H(v)=S} - \frac{1}{d}\sum_{u\in N(v)}\indicator{H(u)=S}\right) e^{-\beta m(\chi)}\,,
\end{equation}
which give constraints for the linear program.

For $L\in\hcL_{4,q}$ and $S\in\mathcal S_{4,q}$, the quantities $\vec c_L$ and $\vec \gamma^S_L$ in~\eqref{eq:cLdef} and \eqref{eq:gamLdef} involve sums over local colourings $\chi:V_L\to[q]$. 
The fact that these sums grow in size with $q$ is a problem for the general case $q\ge5$ of Theorem~\ref{thm:d4}, but again we can solve this by considering permutations of the colours.
By permuting colours we only need consider $L$ whose boundary colours form an initial segment of $\{1,2,\dotsc\}$, and we write $q_L$ for the number of colours used on the boundary of $L$.
Since $|V_L|=d+1$ we only need $q_L+d+1$ colours to represent each equivalence class of local colourings $\chi$.
That is, for any quantity $f(\chi)$ which is invariant under permutations of the colours that are not used on the boundary of $L$, we have
\[
\sum_{\chi: V_L\to[q]}f(\chi) = \sum_{\chi\in\Omega_L}z(\chi)f(\chi)\,,
\]
where $\Omega_L$ is the set of maps $\{\chi: V_L\to [q_L+d+1]\}$ which use an initial segment (possibly empty) of the `extra colours' $\{q_L+1,\dotsc, q_L+d+1\}$, and
\[
z(\chi) = \binom{q-q_L}{\ell(\chi)-q_L}\,,
\]
is the size of the equivalence class of a map $\chi$ that uses largest colour $\ell(\chi)$. 
This idea turns the calculations giving $\vec c$ and $\vec\gamma^S$ into finite sums whose size is bounded independently of $q$.

Let $\vec{p}\in\mathbb{R}^{\hcL_{4,q}}$ be the variables of the primal programs, where $\vec{p}_L$ corresponds to the probability that a local view isomorphic to $L$ is sampled in the random process. 
In standard form, the primal linear programs in we wish to solve are
\begin{align}\label{eq:programs}
\begin{rcases}
U^{\min} = \min\;\\
U^{\max} = \max\;
\end{rcases}
\; \vec{p}\cdot \vec{c} \quad\text{s.t.}\quad A\vec{p} = \vec{b},\quad \vec{p}\ge \vec{0}\,,
\end{align}
where $A$ is the $|\hcL_{4,q}|\times \big(1+|\mathcal S_{4,q}|\big)$ matrix with first row all ones, and subsequent rows given by $\vec{\gamma}^S$ for $S\in \mathcal S_{4,q}$, and $\vec{b}$ is the vector $(1, 0, \dotsc, 0)$ with exactly $|\mathcal S_{4,q}|$ zeros.
Each row of $A$ together with corresponding entry from $\vec b$ corresponds to a constraint: the first row gives $\sum_L\vec{p}_L=1$, and subsequent rows give $\sum_L\vec{p}_L\vec{\gamma}^S_L=0$.
For concreteness, we consider the rows of $A$ and entries of $\vec b$ to be indexed by the set $\{1\}\cup \mathcal S_{4,q}$ with $1$ corresponding to the probability distribution constraint $\sum_L\vec{p}_L=1$ being first.

\subsection{Minimising}

We solve the minimisation program by considering the dual.
Linear programming duality~\cite{boyd2004convex} states that 
\begin{align}
U^{\min} = \max \vec\Delta\cdot \vec b\quad\text{s.t.}\quad A^T\vec \Delta \le \vec c,\quad \vec \Delta\ge \vec 0\,,
\end{align}
where $\vec \Delta$ is a vector of dual variables. 
Let $\pstar$ correspond to the distribution of the random local view in $K_{d,d}$, and $U^*=U^q_{K_{4,4}}(\beta)$. 
We know that $U^*$ is feasible for the primal program, which gives $U^{\min}\le U^*$, and if we find a vector of dual variables $\Dstar$ which are feasible for the dual program such that $\Dstar_1=U^*$, then the dual program shows $U^{\min} \ge \Dstar\cdot\vec b = U^*$ and hence $U^{\min}=U^*$ as required. 
Then it suffices to find a dual-feasible vector $\Dstar$ with $\Dstar_1=U^*$. 

It turns out we can solve for the dual variables uniquely provided we choose the minimal number of constraints.
In general for an optimum supported on $k$ local views, we can set all but $k$ dual variables to zero and solve the subsystem of the dual constraints corresponding to the $k$ local views with equality to obtain the remaining dual variables. 

In $K_{d,d}$ every vertex $v$ has exactly $d-1$ vertices at distance 2 which form the external neighbours of any local view from $v$, and each $u\in N(v)$ is adjacent to every one of these external neighbours, hence the local views which arise in $K_{d,d}$ with positive probability have the same multiset for each $u\in N(v)$. 
Up to isomorphism there are $|\mathcal S_{d-1,q}|$ such local views.
With $d=4$ and any $q\ge 5$ we have $1+|\mathcal S_{4,q}|=6$, but only need $|\mathcal S_{3,q}|=3$ constraints. 
Rather frustratingly there is no set of three constraints that work for the entire range of $q$, we must consider $q=5$ and $q\ge 6$ separately. 
In either case we select the first constraint corresponding to $\sum_L \vec p_L=1$, and two constraints corresponding to $q$-partitions in $\mathcal S_{4,q}$. 
For $q=5$ we use the partitions $4$ and $2+1+1$, and for $q\ge 6$ we use $2+1+1$ and $1+1+1+1$. 
In both cases we solve for the corresponding dual variables and verify dual feasibility.

In each case, we have a set $\mathcal L^*$ of three local views which arise in $K_{4,4}$, the probability distribution constraint, and two $q$-partitions $S_1,S_2\in \mathcal S_{4,q}$ giving the remaining constraints.
We set $\Dstar_1=U^*$, $\Dstar_S=0$ for $S\in\mathcal S_{4,q}\setminus\{S_1,S_2\}$, and find $\Dstar_{S_1}$, $\Dstar_{S_2}$ by solving the dual constraints corresponding to $L\in\mathcal L^*$ with equality.
That is, we take the rows of $A^T\Dstar\le\vec c$ corresponding to $L\in \mathcal L^*$, and solve the system of equations given by
\[
U^* + \Dstar_{S_1}\vec\gamma^{S_1}_L + \Dstar_{S_2}\vec\gamma^{S_2}_L = \vec c_L\qquad\forall L\in\mathcal L^*\,.
\]
Note that we actually have three equations for the two unknowns $\Dstar_{S_1}$ and $\Dstar_{S_2}$, but the system is indeed consistent.

After the vector $\Dstar$ has been found, dual feasibility reduces to the statement that the entries of the \emph{slack vector} $\vec c-A^T\Dstar$ are non-negative for all $L$. 
In fact we show that the optimising distribution $\pstar$ is unique by complementary slackness: we verify that the slack is strictly positive for all $L\notin\mathcal L^*$. 
Since unions of $K_{4,4}$'s are the only graphs that have distribution $\pstar$, these graphs are the only graphs attaining equality in the corresponding bounds in Theorem~\ref{thm:d4} and~\eqref{eq:Cconj}.

We verify this on a computer by substituting $e^\beta=1+t$ and either $q=5$ or $q=r+6$, and showing that for some `magic factor' $M>0$ which is independent of $L$, the quantity
\begin{equation}\label{eq:minobs}
M\cdot \left(\vec c_L - U^* - \Dstar_{S_1}\vec\gamma^{S_1}_L - \Dstar_{S_2}\vec\gamma^{S_2}_L \right)\,,
\end{equation}
is a polynomial with positive coefficients in $t$ or both $t$ and $r$, for all $L$ which do not occur in $K_{4,4}$. 
Note that $\beta>0$ corresponds to the range $t>0$, and $q\ge 6$ corresponds to $r\ge 0$.
The magic factor depends on the case $q=5$ or $q\ge 6$, and was found by hand. For $q=5$ we take
\[
\begin{split}
M = 6 t^2 (t+1)^{31}&(2t^2+6 t+5) (16 t^{10}+176 t^9+888 t^8+2676 t^7+{}\\
   5309t^6&+7260 t^5+6996 t^4+4760 t^3+2224 t^2+660 t+100)\,,
\end{split}
\]
and for $q\ge 6$ we have
\[
\begin{split}
M = 2 (r+3)& (r+4) (t+1)^{35} (r^4 t^9+9 r^4 t^8+36 r^4 t^7+84 r^4 t^6+126 r^4 t^5+{}\\&126
   r^4 t^4+84 r^4 t^3+36 r^4 t^2+9 r^4 t+r^4+11 r^3 t^9+105 r^3 t^8+{}\\&447 r^3 t^7+1117 r^3
   t^6+1809 r^3 t^5+1971 r^3 t^4+1445 r^3 t^3+{}\\&687 r^3 t^2+192 r^3 t+24 r^3+39 r^2 t^9+411
   r^2 t^8+1926 r^2 t^7+{}\\&5286 r^2 t^6+9393 r^2 t^5+11241 r^2 t^4+9090 r^2 t^3+4806 r^2
   t^2+{}\\&1512 r^2 t+216 r^2+51 r t^9+645 r t^8+3477 r t^7+10715 r t^6+{}\\&21096 r t^5+27816 r
   t^4+24812 r t^3+14580 r t^2+5184 r t+{}\\&864 r+18 t^9+342 t^8+2250 t^7+7954 t^6+17508
   t^5+{}\\&25428 t^4+24888 t^3+16200 t^2+6480 t+1296)\,.
\end{split}
\]
Given $M$ it is easy to use a computer algebra system to verify the assertion that~\eqref{eq:minobs} is a polynomial in $t$ with non-negative coefficients, and is zero if and only if $L$ arises in $K_{4,4}$.
Indeed, for each $L$ the verification is straightforward by hand, but requires lengthy calculations more suited to a computer.

\subsection{Maximising}

Note that $K_5$ itself is a valid local view that has no boundary, hence the optimum in the maximisation program is supported on a single local view. 
This means we wish to use a single constraint in the program, namely that $\sum_L p_L=1$. 
It suffices to verify that the largest entry of $\vec c$ corresponds to $L=K_5$, which we do by showing that for another `magic factor' $M'=2*(t + 1)^{25}$, and substitutions $e^b=1+t$, $q=r+5$, the quantity 
\[
M'\cdot\big(\vec c_{K_5} - \vec c_L\big)
\]
is a polynomial in $t$ and $r$ with positive coefficients for all $L\ne K_5$.
Again, this is simple to verify with a computer.

\section{Generating local views}\label{sec:d4computer}

The method of the previous section reduces an optimisation problem over the infinite set of $4$-regular graphs to the verification of a finite number of inequalities, which can be done by computer. 
The main difficulty in establishing Theorem~\ref{thm:d4} is now to enumerate all local views so that this verification can be performed.

\subsection{Representing local views as a simple graph}

Formally, the local view is a graph on $d+1$ vertices where an identified vertex $v$ has degree $d$ and for the other vertices we record a multiset of observed colours. 
For ease of use with standard graph isomorphism algorithms, we represent the entire local view as a simple graph with no coloured vertices. 
A local view already gives us the graph structure on $\{v\}\cup N(v)$, and we add to this graph a set $B$ of pendant boundary vertices to the neighbours of $v$. 
Each neighbour $u$ of $v$ is degree $4$ in the original graph, so we add enough pendant vertices such that this holds in the representation of the local view. 
Now we can represent the multiset of observed colours by colouring the pendant boundary vertices, but we wish to have a simple graph representation that uses no colours. 
So to represent the colours of the boundary we add a set $C$ of auxiliary \emph{colour vertices} and an edge from each $b\in B$ to exactly one $c\in C$ to represent $b$ being coloured by colour $c$. 
Note that with this definition, any local view in $\hcL_{4,q}$ can be represented by with at most $|B|\le 12$ colour vertices. 
See Figure~\ref{fig:Lrep} for the simple graph representation of the local view pictured in Figure~\ref{fig:L}. 
We denote by $\simp L$ this representation of $L\in\hcL_{q,d}$.

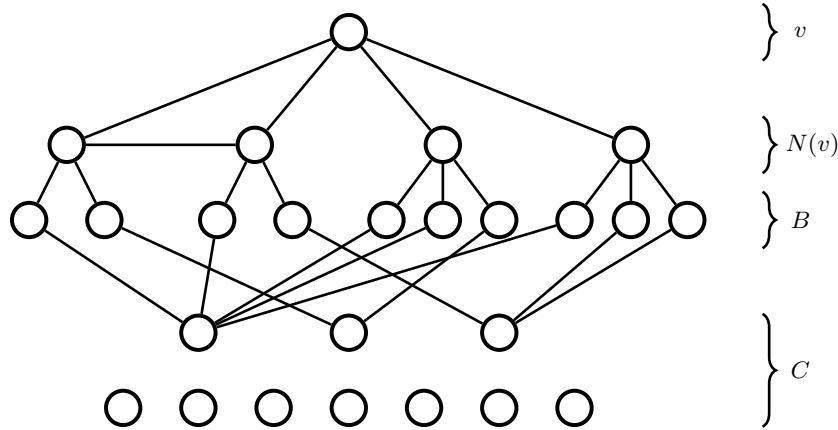
\begin{figure}[h]
\begin{tikzpicture}
	\def\bcs{{{1,2},{1,3},{1,1,2},{1,3,3}}}
	\node (v) at (0,0) {};

	\pgfmathsetmacro{\num}{3}
	\pgfmathsetmacro{\spa}{2}
	\pgfmathsetmacro{\len}{4}
	\foreach \i in {1,...,\num} {
		\pgfmathsetmacro{\spai}{(\i-1-(\num-1)/2)*\spa}
		\node (c\i) at ($ (\spai,-\len) $) {};
	}
	\pgfmathsetmacro{\num}{7}
	\pgfmathsetmacro{\spa}{1}
	\pgfmathsetmacro{\len}{5}
	\foreach \i in {1,...,\num} {
		\pgfmathsetmacro{\spai}{(\i-1-(\num-1)/2)*\spa}
		\pgfmathsetmacro{\j}{int(\i+3)}
		\node at ($ (\spai,-\len) $) {};
	}
	
	\pgfmathsetmacro{\num}{4}
	\pgfmathsetmacro{\spa}{2.5}
	\pgfmathsetmacro{\len}{1.5}
	\foreach \i in {1,...,\num} {
		\pgfmathsetmacro{\spai}{(\i-1-(\num-1)/2)*\spa}
		\node (u\i) at ($ (\spai,-\len) $) {};
		\draw (v) -- (u\i);
	}
	\draw (u1) -- (u2);	
	
	\foreach \i in {1,2} {
    	\pgfmathsetmacro{\num}{2}
    	\pgfmathsetmacro{\spa}{1}
    	\pgfmathsetmacro{\len}{1}
    	\foreach \j in {1,...,\num} {
    		\pgfmathsetmacro{\spai}{(\j-1-(\num-1)/2)*\spa}
			\pgfmathsetmacro{\c}{\bcs[\i-1][\j-1]}
    		\node (u\i\j) at ($ (u\i) + (\spai,-\len) $) {};
    		\draw (u\i) -- (u\i\j);
			\draw (u\i\j) -- (c\c);
    	}
	}

	\foreach \i in {3,4} {
    	\pgfmathsetmacro{\num}{3}
    	\pgfmathsetmacro{\spa}{0.75}
    	\pgfmathsetmacro{\len}{1}
    	\foreach \j in {1,...,\num} {
    		\pgfmathsetmacro{\spai}{(\j-1-(\num-1)/2)*\spa}
			\pgfmathsetmacro{\c}{\bcs[\i-1][\j-1]}
    		\node (u\i\j) at ($ (u\i) + (\spai,-\len) $) {};
    		\draw (u\i) -- (u\i\j);
			\draw (u\i\j) -- (c\c);
    	}
	}
	\draw[decorate,decoration={brace,amplitude=5pt}]
			(5.5,0.375) -- (5.5,-0.375) node [draw=none,midway,anchor=west,xshift=0.25cm] {\footnotesize $v$};
	\draw[decorate,decoration={brace,amplitude=5pt}]
			(5.5,-1.125) -- (5.5,-1.875) node [draw=none,midway,anchor=west,xshift=0.25cm] {\footnotesize $N(v)$};
	\draw[decorate,decoration={brace,amplitude=5pt}]
			(5.5,-2.125) -- (5.5,-2.875) node [draw=none,midway,anchor=west,xshift=0.25cm] {\footnotesize $B$};
	\draw[decorate,decoration={brace,amplitude=5pt}]
			(5.5,-3.75) -- (5.5,-5.25) node [draw=none,midway,anchor=west,xshift=0.25cm] {\footnotesize $C$};
\end{tikzpicture}
\caption{A representation of the same local view as in Figure~\ref{fig:L}, but where colours on boundary vertices are indicated by edges to auxiliary colour vertices. There are 10 boundary vertices, hence we draw 10 colour vertices, 7 of which are unused.\label{fig:Lrep}}
\end{figure}

We will enumerate the simple representations of local views as follows.
\begin{enumerate}
\item
Start with a copy of $K_{1,d}$ which is the vertex $v$ and its $d$ neighbours. 
\item
Add any collection of edges inside $N(v)$, and add pendant boundary vertices to each $u\in N(v)$ such that each $u\in N(v)$ has degree $d$.
\item
Add a set of colour vertices $C$ of size $|B|$, and colour the boundary vertices by adding, for each $b\in B$, exactly one edge to a vertex in $C$.
\end{enumerate}
Provided the possible edges inside $N(v)$ and the possible colourings are enumerated exhaustively in steps~(2) and~(3), this process gives at least one representative of every possible local view and is easy to implement on a computer. 
Note that by virtue of being represented on a computer we are forced to consider a labelling of the vertices, so a naive implementation of step~(2) involves generating every \emph{labelled graph} on $d$ vertices. 
Similarly, a naive implementation of step~(3) will consider all functions from the boundary $B$ to the colour vertices $C$. 
This is highly inefficient\footnote{For example, when $d=4$ each triangle-free local view has $12$ boundary vertices to colour with up to $12$ colours.
There are $12^{12}=8\,916\,100\,448\,256$ such colourings, but in fact only $1636$ isomorphism classes of triangle-free local view.} since the quantities we are interested in are independent of the labelling of $N(v)$ and invariant under permutations of the colours. 
Exploiting these symmetries drastically reduces the computation required for steps~(2) and~(3).

Using the simple graph representation, the generation of $\hcL_{d,q}$ for any $d$ is now an instance of a well-studied general problem: we want a single labelled representative of each isomorphism class of some unlabelled combinatorial objects (the graphs $\simp L$ for $L\in\hcL_{d,q}$) under the action of some symmetry group (permuting the neighbours of $v$ and permuting colours). 
In this paper we use a canonical generation algorithm of McKay~\cite{MCKAY1998306} to achieve this efficiently, and an advantage of the simple graph representation is that all the symmetries we must consider arise as graph automorphisms, facilitating the use of pre-existing routines for such problems.

\subsection{Generating graphs}

We first give an efficient version of step (2) of the local view generation algorithm. 
We want to generate a single labelled version of each graph on $d$ vertices in order to enumerate the possible underlying graphs of the local views. 

There is a natural `construction tree' for this process.
The root is the empty labelled graph on $d$ vertices, and the children of any node $X$ are the labelled graphs obtained from $X$ by adding one edge. 
Every leaf in the tree is therefore the complete labelled graph on $d$ vertices. 
McKay's method works by generating only those $X$ that have a `canonical construction path' in this tree. 
A key ingredient is a \emph{canonical labelling function} that, given a labelled graph, outputs an isomorphic labelled graph, and which is constant on equivalence classes of labelled graphs. 
Such a function can be used to identify a single, canonical representative of an unlabelled graph $G$: given an arbitrary initial labelling of $G$, the function produces a fixed labelling that depended only on $G$ and not the initial labelling. 
These functions are well-studied, since any such function solves the \emph{graph isomorphism problem}: testing whether two graphs are isomorphic reduces to checking equality of their corresponding canonical labellings.

For completeness, we describe a basic example of a canonical labelling function due to Read~\cite{read1978every}, and independently Farad\v{z}ev~\cite{faradzev1978generation}, here. 
For practical purposes we use a more efficient algorithm of McKay~\cite{mckay1981practical} which exploits much more graph structure, specifically the implementation supplied with Sage 8.0. 
The Read--Farad\v{z}ev function relates each labelled graph with vertex set $[n]$ to the binary sequence $a(G)$ of length $\binom{n}{2}$, where $a(G)_i=1$ if the $i$-th pair in the lexicographic ordering of $\binom{[n]}{2}$ is an edge of $G$, and $a(G)_i=0$ otherwise. 
This gives an ordering $\prec$ where $G\prec H$ when $a(G)< a(H)$ in the lexicographic ordering of binary sequences. 
The canonical labelling of $G$ is then the labelling which is maximal in the ordering $\prec$, and can be computed by checking every permutation of the vertices of $G$.

Armed with a canonical labelling function, we give in Figure~\ref{alg:canaug} a short description of McKay's algorithm for canonical graph generation. 
When \Call{scan}{$\overline K_d$} is called, the algorithm explores a subtree of the construction tree for labelled graphs on $d$ vertices and outputs exactly one labelled copy of each graph on $d$ vertices.

\newcommand*{\Yield}{\State\textbf{yield} }
\newcommand*{\Let}{\State\textbf{let} }
\newcommand*{\cLet}{\State\hphantom{\textbf{let} }}
\begin{figure}[htb]
  \begin{algorithmic}[0]
    \Statex
    \Procedure{scan}{$X$}
      \Yield $X$
      \If{$X=K_d$} \Return\EndIf
      \Let $U^*$ contain one representative of each isomorphism class
      \cLet of non-edges of $X$ under the action of $\aut(X)$
      \For{$e \in U^*$}
      	\Let $Y= X + e$, and $Y^*$ be the canonical labelling of $Y$
        \Let $y^*$ be the last edge of $Y^*$, and $y$ be its relabelling in $Y$
        \Let $m(Y)$ be the orbit of $y$ under the action of $\aut(Y)$
      	\If{$e\in m(Y)$}
			\Call{scan}{$Y$}
		\EndIf
      \EndFor
    \EndProcedure
  \end{algorithmic}
\caption{A canonical augmentation algorithm which yields a single labelled representative of all unlabelled graphs on $d$ vertices.\label{alg:canaug}}
\end{figure}

As one can see from the algorithm, an explored node $X$ is output (yielded) and then a set of representatives $U^*$ of possible `next edges' is computed. 
This can be done by computing the orbit of every possible next edge under the action of $\aut(X)$, and picking the least edge of each orbit in the lexicographic order on $\binom{[d]}{2}$. 
Here is the first use of symmetry to prune the construction tree: we only consider one representative of each equivalence class of `next edge' $e$. 
The next step is to ensure that we only recursively explore the subtree rooted at $Y=X+e$ when $e$ is, up to isomorphism, the `last edge' of $Y$. 
Here we take the `last edge' of $Y$ to be the last edge of the canonical labelling $Y^*$ of $Y$ in the lexicographic order on $\binom{[d]}{2}$.
Thus, we observe that the algorithm explores a subtree of the construction tree. 
We omit a more detailed analysis of the algorithm that would constitute a full proof that that exactly one labelled copy of each unlabelled graph is generated, see~\cite{MCKAY1998306} for details.

\subsection{Colouring graphs}

The same algorithm can be used to generate colourings by restricting the edges one adds to the graph. 
For local views, the idea is to start with the necessary graph structure on $\{v\}\cup N(v)\cup B$ with the colour vertices $C$ added as isolated vertices. 
Then colouring a vertex $b\in B$ with colour $c\in C$ corresponds to adding an edge between $b$ and $c$ in the graph. 
For brevity we say a vertex in $B$ is coloured to mean it has exactly one edge from it to a vertex in $C$.

This again gives a construction tree: the root is a graph of the form $\simp L$ but with no edges from $B$ to $C$, and the children of each node are the graphs obtained by adding exactly one edge from an uncoloured boundary vertex to a colour vertex.
The algorithm in Figure~\ref{alg:canaug} can be slightly modified to output produce exactly one labelled representative of every possible colouring of the initial graph structure. 
Here we write $\aut(X)$ to mean the subgroup of graph automorphisms of the local view representation $X$ that fixes $\{v\}$, $N(v)$, $B$, and $C$ as sets.

\begin{figure}[htb]
  \begin{algorithmic}[0]
    \Statex
    \Procedure{colour}{$X$}
      
      \If{all vertices in $B$ are coloured}
      \Yield $X$
      \State\Return
      \EndIf
      \Let $U^*$ contain one representative of each isomorphism class
      \cLet of non-edges from uncoloured boundary vertices to colour
      \cLet vertices under the action of the subgroup of $\aut(X)$
      \For{$e \in U^*$}
      	\Let $Y= X + e$, and $Y^*$ be the canonical labelling of $Y$
        \Let $y^*$ be the last edge of $Y^*$, and $y$ be its relabelling in $Y$
        \Let $m(Y)$ be the orbit of $y$ under the action of $\aut(Y)$
      	\If{$e\in m(Y)$}
			\Call{colour}{$Y$}
		\EndIf
      \EndFor
    \EndProcedure
  \end{algorithmic}
\caption{A canonical augmentation algorithm for colouring the boundary of local views\label{alg:col}}
\end{figure}

The changes compared to Figure~\ref{alg:canaug} are that yielding $X$ is only done when all boundary vertices are coloured, and that the edges we consider adding to $X$ are only those from uncoloured boundary vertices to $C$.

\subsection{Notes}

In this section we justify some features of the above computational methods that might seem peculiar to the experienced user of the algorithms of McKay. 

Our representation of local view seems wasteful in that to represent the externally observed colours for each $u\in N(v)$ we could omit the pendant boundary vertices and record edges from $u$ to auxiliary colour vertices. 
To obtain this representation from the one given above one simply contracts the pendant edges from $u\in N(v)$ to its adjacent boundary vertices. 
This would be a more compact representation, but we would require in some cases multiple edges from the neighbours of $v$ to auxiliary colour vertices. 
For ease of use with existing implementations of algorithms and data structures that do not allow for multiple edges between two vertices in a graph, we chose the representation given above. 

In addition, McKay's work~\cite{mckay1981practical} on isomorphisms naturally allows for coloured vertices to be considered, and can be used to find automorphisms of labelled graphs that preserves the colour classes as sets. 
The reason we do not use this facility directly here is that in the algorithms of McKay one does not allow automorphisms that exchange entire colour classes. 
The trick we use here with auxiliary colour vertices that can be exchanged (amongst themselves) is a standard way of getting around this problem. 
When calling McKay's algorithm as in Figure~\ref{alg:col}, we use four McKay-style non-interchangeable colour classes $\{v\}$, $N(v)$, $B$, and $C$, so that the group $\aut(X)$ computed there fixes these sets, as needed for our application.

\section{Conclusions}\label{sec:conc}

This paper demonstrates that the approach conjectured in~\cite{davies2016extremes} works for $d=4$, and gives reasonably efficient implementations of computationally expensive tasks that are required for the proof. 
The fact that the minimisation program used here (and first discussed in~\cite{davies2016extremes}) has optimum given by $K_{4,4}$ as desired is some evidence that the approach may work in general for $d\ge 3$ and $q\ge d+1$. 

We leave for further work the tasks of trying $d=5$ with this method, noting that the computational techniques described here make no essential use of $d=4$ and can be applied to larger $d$ immediately, albeit slowly.
For completeness, we repeat the relevant conjecture from~\cite{davies2016extremes} here, noting that it is open for $d\ge 5$.

\begin{conj}[Davies, Jenssen, Perkins, Roberts]
The minimisation program~\eqref{eq:programs} (see Section~\ref{sec:d4proof}) has optimum given by $K_{d,d}$ for $d\ge 3$, $q\ge d+1$ and $\beta>0$, which shows that $U^q_{K_{d,d}}(\beta)\le U^q_{G}(\beta)$ for all $d$-regular graphs $G$.
\end{conj}


For the proof of the case $d=4$ of this conjecture we had to solve several difficulties which were not present in the case $d=3$. 
Firstly, the computations required to solve the problem with this method require substantial computing power, and were only feasible due to the careful choice of efficient algorithms such as \Call{scan}{} to generate $\hcL_{d,q}$.
We are able to compute that (for $q\ge 20$) $\hcL_{5,q}$ contains approximately $7$ million local views, compared to only $3529$ in $\hcL_{4,q}$. 
It seems that even computing the coefficients for the linear program for $d=5$ is computationally infeasible today. 
Secondly, there is not a canonical choice of the minimal set of the $\mathcal S_{4,q}$ constraints that gives the correct result for all $q\ge 5$, and no apparent reason why some subsets (of the correct size) work, and some do not for $q=5$ or for $q\ge 6$. 
It took a substantial amount of experimentation to find the proofs presented here, and for larger values of $d$ such experiments would require even more expensive computation.

Fortunately, there is still hope that these computations might be avoidable.
In~\cite{davies2017independent} a general solution for an analogous problem on matchings was given, with an inductive description (in $d$) of the dual variables and proof that the $K_{d,d}$ yields a dual feasible solution. 
For such a proof here, one would need to find a description of the minimal set of constraints to use, solve for the dual variables, and show dual feasibility in general. 
These appear to be challenging tasks, but now that solutions exist for $d\in\{3,4\}$ perhaps it will be possible to spot the beginnings of a pattern.

It would be interesting to settle the $4$-regular case of the conjecture of Galvin and Tetali~\cite{galvin2004weighted} in full; one still has to prove that~\eqref{eq:Cconj} holds for all $4$-regular $G$ when $q\in\{3,4\}$ (though this is known for bipartite graphs~\cite{galvin2004weighted}). 
The linear program given here does not give $U^{\min}=U^q_{K_{4,4}}(\beta)$ for such $q$ and large $\beta$, hence one must try a stronger program, or perhaps a different approach. 
There are natural ways to strengthen the methods given here, such as adding more constraints (if they exist), or taking larger local views: e.g.\ to include vertices at distance at most $2$ from $v$ with the colours of vertices at exactly distance $3$ revealed to form the boundary.

\section{Acknowledgements}

The author would like to thank Will Perkins and Matthew Jenssen for their careful reading of a draft of this article, and an anonymous referee for comments.
The research leading to these results has received funding from the European Research Council under the European Union's Seventh Framework Programme (FP7/2007-2013) / ERC grant agreement \textnumero{} 339109.

\bibliography{paper2}
\bibliographystyle{habbrv}

\end{document}